\documentclass[oneside,english]{amsart}
\usepackage[T1]{fontenc}
\usepackage[latin9]{inputenc}
\usepackage[a4paper]{geometry}
\usepackage{prettyref}
\usepackage{mathrsfs}
\usepackage{amstext}
\usepackage{amsthm}
\usepackage{amssymb}
\usepackage[numbers]{natbib}
\usepackage[all]{xy}

\makeatletter
\numberwithin{equation}{section}
\numberwithin{figure}{section}
\theoremstyle{plain}
\newtheorem{thm}{\protect\theoremname}[section]
\theoremstyle{definition}
\newtheorem{example}[thm]{\protect\examplename}
\theoremstyle{plain}
\newtheorem{cor}[thm]{\protect\corollaryname}
\theoremstyle{remark}
\newtheorem{rem}[thm]{\protect\remarkname}
\theoremstyle{plain}
\newtheorem{lem}[thm]{\protect\lemmaname}
\theoremstyle{remark}
\newtheorem*{claim*}{\protect\claimname}
\theoremstyle{definition}
\newtheorem{problem}[thm]{\protect\problemname}

\usepackage{mathtools}
\usepackage{braket}

\newcommand{\ns}[1]{\prescript{\ast}{}{#1}}

\newcommand{\INF}{\operatorname{INF}}
\newcommand{\IPC}{\operatorname{IPC}}
\newcommand{\Ends}{\operatorname{Ends}}
\newcommand{\Cay}{\Gamma}
\newcommand{\sgn}{\operatorname{sgn}}

\newcommand{\pCoarse}{\mathbf{Coarse}_{\ast}}

\newcommand{\id}{\operatorname{id}}
\newcommand{\pMetr}{\mathbf{Metr}_{\ast}}
\newcommand{\Sets}{\mathbf{Sets}}
\newcommand{\To}{\mathrel{\Rightarrow}}

\newrefformat{cor}{Corollary \ref{#1}}
\newrefformat{rem}{Remark \ref{#1}}

\makeatother

\usepackage{babel}
\providecommand{\claimname}{Claim}
\providecommand{\corollaryname}{Corollary}
\providecommand{\examplename}{Example}
\providecommand{\lemmaname}{Lemma}
\providecommand{\problemname}{Problem}
\providecommand{\remarkname}{Remark}
\providecommand{\theoremname}{Theorem}

\begin{document}
\title{Sequential ends and nonstandard infinite boundaries of coarse spaces}
\thanks{The presentation of this work was supported by the Research Institute
for Mathematical Sciences, an International Joint Usage/Research Center
located in Kyoto University.}
\author{Takuma Imamura}
\thanks{The author was supported by the Morikazu Ishihara (Shikata) Research
Encouragement Fund and by JST ERATO HASUO Metamathematics for Systems
Design Project (No. JPMJER1603).}
\address{Research Institute for Mathematical Sciences\\
Kyoto University\\
Kitashirakawa Oiwake-cho, Sakyo-ku, Kyoto 606-8502, JAPAN}
\email{timamura@kurims.kyoto-u.ac.jp}
\begin{abstract}
This paper is an addendum to the author's previous paper \citep{Ima20a}.
\citet{MSM10} introduced a functor $\sigma\colon\pCoarse\to\Sets$,
where $\pCoarse$ is the category of pointed coarse spaces and coarse
maps. \citet{DLT13} introduced a functor $\varepsilon\colon\pCoarse\to\Sets$,
and proved that $\varepsilon$ coincides with $\sigma$ on $\pMetr$
(the full subcategory of metrisable spaces). Using techniques of nonstandard
analysis, the author in \citep{Ima20a} provided a functor $\iota\colon\mathscr{C}\subseteq\pCoarse\to\Sets$,
where $\mathscr{C}$ is an arbitrary small full subcategory, and a
natural transformation $\omega\colon\sigma\restriction\mathscr{C}\To\iota$.
The surjectivity of $\omega$ has been proved for all proper geodesic
metrisable spaces, while the injectivity has remained open. In this
note, we first pointed out that $\omega$ is the composition of two
natural transformations $\varphi\restriction\mathscr{C}\colon\sigma\restriction\mathscr{C}\To\varepsilon\restriction\mathscr{C}$
and $\omega'\colon\varepsilon\restriction\mathscr{C}\To\iota$, and
then show that $\omega'$ is injective for all spaces in $\mathscr{C}$.
As a corollary, $\omega$ is injective for all metrisable spaces in
$\mathscr{C}$. This partially answers some of the problems posed
in \citep{Ima20a}.
\end{abstract}

\maketitle

\section{Introduction}

Small-scale topology can be considered as the study of infinitesimal
structures of spaces, maps, homotopies, and so forth from the point
of view of nonstandard analysis. On the other hand, large-scale topology
is of infinite (or finite) structures (cf. \citet[p. 7]{PZ07}). The
notion of a coarse space introduced by \citet{Roe03} plays a central
role in large-scale topology, as the notion of a uniform space does
in small-scale topology. Recall that a \emph{coarse structure} on
a set $X$ is an ideal $\mathcal{C}_{X}$ of the poset $\mathcal{P}\left(X\times X\right)$
(with respect to the inclusion) that fulfills the following axioms:
\begin{enumerate}
\item $\Delta_{X}:=\set{\left(x,x\right)|x\in X}\in\mathcal{C}_{X}$;
\item $E\circ F:=\set{\left(x,y\right)|\left(x,z\right)\in E\text{ and }\left(z,y\right)\in F\text{ for some }z\in X}\in\mathcal{C}_{X}$
for all $E,F\in\mathcal{C}_{X}$;
\item $E^{-1}:=\set{\left(x,y\right)|\left(y,x\right)\in E}\in\mathcal{C}_{X}$
for all $E\in\mathcal{C}_{X}$.
\end{enumerate}
The set $X$ together with the coarse structure $\mathcal{C}_{X}$
is called a \emph{coarse space}. The elements of $\mathcal{C}_{X}$
are called \emph{controlled sets} or \emph{entourages}. A subset $B$
of $X$ is said to be \emph{bounded} if $B\times B$ is a controlled
set.
\begin{example}
Let $d_{X}\colon X\times X\to\mathbb{R}_{\geq0}\cup\set{+\infty}$
be a (generalised) metric on a set $X$. The family
\[
\mathcal{C}_{X}:=\set{E\subseteq X\times X|\sup d_{X}\left(E\right)<+\infty}
\]
forms a coarse structure on $X$, and is called the \emph{bounded
coarse structure} induced by $d_{X}$. A subset of $X$ is bounded
with respect to $\mathcal{C}_{X}$ if and only if it is bounded with
respect to $d_{X}$. Throughout the paper, we assume that every metric
space is endowed with the bounded coarse structure.
\end{example}

Let $X_{i}\ \left(i=0,1\right)$ be coarse spaces. A map $f\colon X_{0}\to X_{1}$
is said to be
\begin{enumerate}
\item \emph{proper} if $f^{-1}\left(B\right)$ is bounded in $X_{0}$ for
all bounded subsets $B$ of $X_{1}$;
\item \emph{bornologous} if $\left(f\times f\right)\left(E\right):=\set{\left(f\left(x\right),f\left(y\right)\right)|\left(x,y\right)\in E}\in\mathcal{C}_{Y}$
for all $E\in\mathcal{C}_{X}$.
\item \emph{coarse} if it is proper and bornologous.
\end{enumerate}
Denote by $\pCoarse$ the category of pointed coarse spaces and basepoint-preserving
coarse maps; and by $\pMetr$ its full subcategory of metrisable spaces.

Coarse maps $f,g\colon X_{0}\to X_{1}$ are said to be \emph{close}
if $\left(f\times g\right)\left(\Delta_{X_{0}}\right)=\set{\left(f\left(x\right),g\left(x\right)\right)|x\in X}\in\mathcal{C}_{X_{1}}$.
A coarse map $f\colon X_{0}\to X_{1}$ is called a \emph{coarse equivalence}
if there exists a coarse map (called a coarse inverse) $g\colon X_{1}\to X_{0}$
such that $f\circ g$ and $g\circ f$ are close to the identity maps
$\id_{X_{1}}$ and $\id_{X_{0}}$, respectively. The closeness relation
of (basepoint-preserving) coarse maps gives a congruence on the category
$\pCoarse$.

\subsection{Invariant $\sigma$}

\citet{MSM10} and \citet{DLW11} introduced a set-valued coarse invariant
$\sigma\left(X,\xi\right)$ of a pointed \emph{metric} space $\left(X,\xi\right)$.
Recall the simplified definition given by \citet{DLT13}. Let $\left(X,\xi\right)$
be a pointed \emph{coarse} space. (This generalisation to coarse spaces
by \citep{Ima20a} is merely a rewriting of the original definition
in terms of coarse structures.) A coarse map $s\colon\left(\mathbb{N},0\right)\to\left(X,\xi\right)$
is called a \emph{coarse sequence} in $\left(X,\xi\right)$. Notice
that a coarse map is precisely a sequence $\set{s\left(i\right)}_{i\in\mathbb{N}}$
in $X$ with the following properties:
\begin{description}
\item [{properness}] the sequence $\set{s\left(i\right)}_{i\in\mathbb{N}}$
diverges from $\xi$ to infinity, i.e., for all bounded subsets $B\ni\xi$
of $X$, there is an $N\in\mathbb{N}$ such that $\set{s\left(i\right)}_{i\geq N}$
is contained in $X\setminus B$;
\item [{bornologousness}] there is an $E\in\mathcal{C}_{X}$ such that
$\left(s\left(i\right),s\left(i+1\right)\right)\in E$ for all $i\in\mathbb{N}$.
\end{description}
Denote the set of all coarse sequences in $\left(X,\xi\right)$ by
$S\left(X,\xi\right)$. Given $s,t\in S\left(X,\xi\right)$, we say
that $s$ and $t$ are \emph{$\sigma$-equivalent} ($s\equiv_{X,\xi}^{\sigma}t$)
if there exists a finite sequence $\set{s_{i}}_{i\leq n}$ in $S\left(X,\xi\right)$
such that $s_{0}=s$, $s_{n}=t$, and either $s_{i}$ or $s_{i+1}$
is a subsequence to the other. The quotient set
\[
\sigma\left(X,\xi\right):=S\left(X,\xi\right)/\equiv_{X,\xi}^{\sigma}
\]
can be extended to a functor $\sigma\colon\pCoarse\to\Sets$. More
precisely, the morphism part is given by
\[
\sigma f\left[s\right]_{\equiv_{X,\xi}^{\sigma}}:=\left[f\circ s\right]_{\equiv_{Y,\eta}^{\sigma}}.
\]
Moreover, $\sigma$ is invariant under coarse equivalences (relative
to the base points), i.e., if two coarse maps $f,g\colon\left(X,\xi\right)\to\left(Y,\eta\right)$
are close, then $\sigma f=\sigma g$. See \citep{Ima20b} for the
systematic proofs of the functoriality and the coarse invariance of
$\sigma$.

The calculation of the invariant $\sigma\left(X,\xi\right)$ is quite
hard because of the difficulty of the decision of $s\equiv_{X,\xi}^{\sigma}t$.
On the other hand, the coarse invariants $\varepsilon\left(X,\xi\right)$
and $\iota\left(X,\xi\right)$ defined below can easily and intuitively
be calculated.

\subsection{Invariant $\varepsilon$}

\citet{DLT13} provided alternative definition of $\sigma$ in terms
of sequential ends. Before recalling the definition of sequential
ends, we first recall the definition of ends of a topological space
introduced by \citet{Fre31}. Let $X$ be a topological space. A proper
continuous map $r\colon\mathbb{R}_{+}\to X$ is called a \emph{proper
ray}. (A map between topological spaces is said to be \emph{proper}
if the inverse image of each compact set is compact.) Two proper rays
$r_{0},r_{1}$ in $X$ are said to be \emph{converge to the same end}
if for all compact subsets $K$ of $X$ there exists an $N\in\mathbb{N}$
such $r_{0}\left(\mathbb{R}_{\geq N}\right)$ and $r_{1}\left(\mathbb{R}_{\geq N}\right)$
are contained in the same path-connected component of $X\setminus K$.
This gives an equivalence relation of proper rays in $X$. The quotient
set is denoted by $\Ends$$\left(X\right)$. See also \citet[pp .144--148]{BH99}.
\begin{thm}[{\citet[Proposition 2.12]{DLT13}}]
Let $\left(X,\xi\right)$ be a proper geodesic pointed metric space.
Then $\Ends\left(X\right)$ is isomorphic to $\sigma\left(X,\xi\right)$.
\end{thm}

Returning to the definition of sequential ends, let $\left(X,\xi\right)$
be a pointed coarse space. Given $s,t\in S\left(X,\xi\right)$, we
say that $s$ and $t$ are \emph{$\varepsilon$-equivalent} ($s\equiv_{X,\xi}^{\varepsilon}t$)
if there is an $E\in\mathcal{C}_{X}$ such that for all bounded subsets
$B\ni\xi$ of $X$ there exists an $N\in\mathbb{N}$ such that $\set{s\left(i\right)}_{i\geq N}$
and $\set{t\left(i\right)}_{i\geq N}$ are contained in the same $E$-connected
component of $X\setminus B$. Note that a subset $A$ of $X$ is \emph{$E$-connected}
if and only if for any two points $x,y\in A$, there is a finite sequence
$\set{x_{i}}_{i\leq n}$ in $A$, called an \emph{$E$-chain}, such
that $x_{0}=x$, $x_{n}=y$, and $\left(x_{i},x_{i+1}\right)\in E$
for all $i<n$. The $\equiv_{X,\xi}^{\varepsilon}$-equivalence classes
are called \emph{sequential ends} of $\left(X,\xi\right)$. The quotient
set
\[
\varepsilon\left(X,\xi\right):=S\left(X,\xi\right)/\equiv_{X,\xi}^{\varepsilon}
\]
can be considered as a coarsely invariant functor $\varepsilon\colon\pCoarse\to\Sets$.
The morphism part of $\varepsilon$ is defined similarly to that of
$\sigma$.
\begin{thm}[{\citet[Theorem 3.3]{DLT13}}]
\label{thm:sigma-implies-epsilon}Let $\left(X,\xi\right)$ be a
pointed coarse space. Then $\equiv_{X,\xi}^{\sigma}$ implies $\equiv_{X,\xi}^{\varepsilon}$.
\end{thm}

\begin{proof}
Obvious from the fact that every coarse sequence is $\varepsilon$-equivalent
to its subsequences.
\end{proof}
\begin{cor}
\label{cor:Phi}Let $\left(X,\xi\right)$ be a pointed coarse space.
The map $\varphi_{\left(X,\xi\right)}\colon\sigma\left(X,\xi\right)\to\varepsilon\left(X,\xi\right)$
defined by
\[
\varphi_{\left(X,\xi\right)}\left[s\right]_{\equiv_{X,\xi}^{\sigma}}:=\left[s\right]_{\equiv_{X,\xi}^{\varepsilon}}
\]
is well-defined, surjective and natural in $\left(X,\xi\right)$.
\end{cor}

On the other hand, the converse of \prettyref{thm:sigma-implies-epsilon}
holds for metrisable spaces.
\begin{thm}[{\citet[Theorem 3.3]{DLT13}}]
\label{thm:E-implies-S}Let $\left(X,\xi\right)$ be a pointed metric
space. Then $\equiv_{X,\xi}^{\varepsilon}$ implies $\equiv_{X,\xi}^{\sigma}$.
\end{thm}

\begin{proof}
Let $s,t\in S\left(X,\xi\right)$ and suppose $s\equiv_{X,\xi}^{\varepsilon}t$.
Let $K>0$ be a witness of the $\varepsilon$-equivalence. Then for
each $r\in\mathbb{N}$, there is an $N_{r}\in\mathbb{N}$ such that
$\set{s\left(i\right)}_{i\geq N_{r}}$ and $\set{t\left(i\right)}_{i\geq N_{r}}$
are contained in the same $K$-connected component outside the $r$-ball
$B_{X}\left(\xi;r\right)$. In particular, there exists a $K$-chain
$\set{u_{i}^{r}}_{i\leq M_{r}}$ that connects $s\left(N_{r}\right)$
and $t\left(N_{r}\right)$ outside $B_{X}\left(\xi;r\right)$. We
may assume that $N_{r}<N_{r+1}$ for all $r\in\mathbb{N}$. Then the
concatenation of the sequences
\begin{gather*}
\set{s\left(i\right)}_{i\leq N_{0}},\set{u_{i}^{0}}_{i\leq M_{0}},\set{t\left(N_{0}-i\right)}_{i\leq N_{0}}\\
\set{t\left(i\right)}_{i\leq N_{1}},\set{u_{M_{1}-i}^{1}}_{i\leq M_{1}},\set{s\left(M_{1}-i\right)}_{i\leq M_{1}-M_{0}},\\
\set{s\left(M_{0}+i\right)}_{i\leq M_{2}-M_{0}},\set{u_{i}^{2}}_{i\leq M_{2}},\set{t\left(M_{2}-i\right)}_{i\leq M_{2}-M_{1}},\\
\vdots
\end{gather*}
is proper bornologous, and has $s$ and $t$ as subsequences, whence
$s\equiv_{X,\xi}^{\sigma}t$.
\end{proof}
\begin{cor}
\label{cor:phi-bijective}The functors $\sigma$ and $\varepsilon$
are equal on $\pMetr$. The natural map $\varphi_{\left(X,\xi\right)}\colon\sigma\left(X,\xi\right)\to\varepsilon\left(X,\xi\right)$
is therefore the identity map for all metrisable spaces $\left(X,\xi\right)$.
\end{cor}

\begin{rem}
\label{rem:remark-on-phi}In the proof of \prettyref{thm:E-implies-S},
the properness of the resulting sequence depends on the fact that
each bounded subset $B\ni\xi$ of $X$ is contained in some $r$-ball
$B\left(\xi;r\right)$. In other words, the bornology of $X$ is generated
by the countable family $\set{B\left(\xi;r\right)|r\in\mathbb{N}}$.
Such a countability property fails in general pointed coarse spaces.
\end{rem}

\subsection{Invariant $\iota$}

In \citep{Ima20a}, we introduced a set-valued invariant $\iota\left(X,\xi\right)$
of a (standard) pointed coarse space $\left(X,\xi\right)$ via nonstandard
analysis. First, let us recall the nonstandard treatment of coarse
spaces. We assume the reader to be familiar with the terminology of
nonstandard analysis. Fix a small full subcategory $\mathscr{C}$
of $\pCoarse$. By the reflection principle, there exists a transitive
set $\mathbb{U}\ni\mathcal{C}$, the \emph{standard universe}, such
that all classical objects we need belong to $\mathbb{U}$ and all
(but finitely many) set-theoretic formulae we need are absolute with
respect to $\mathbb{U}$. We also fix a sufficiently saturated elementary
extension $\ast\colon\mathbb{U}\hookrightarrow\ns{\mathbb{U}}$; $x\mapsto\ns{x}$,
the \emph{nonstandard extension}. See \citet[Section 4.4]{CK90} for
more detailed account of nonstandard analysis.

Let $X$ be a standard coarse space. The \emph{finite proximity} of
$\ns{X}$ is the binary relation on $\ns{X}$ defined by
\[
x\sim_{X}y:\iff\exists E\in\mathcal{C}_{X}.\left(x,y\right)\in\ns{E}.
\]
For each $\xi\in\ns{X}$, the set of all points infinitely far away
from $\xi$ is denoted by
\[
\INF\left(X,\xi\right):=\set{x\in\ns{X}|x\nsim_{X}\xi}.
\]
The finite proximity relation $\sim_{X}$ completely characterises
the coarse structure of $X$. We have indeed the following characterisations
(see \citep{Ima19}).
\begin{enumerate}
\item A subset $E$ of $X\times X$ is controlled if and only if $\ns{E}\subseteq{\sim_{X}}$.
\item A subset $B$ of $X$ is bounded if and only if $\ns{B}\subseteq\ns{X}\setminus\INF\left(X,\xi\right)$
for all $\xi\in B$.
\item A subset $A$ of $X$ is $E$-connected for some $E\in\mathcal{C}_{X}$
if and only if for any two points $x,y\in\ns{A}$, there is an internal
hyperfinite sequence $\set{x_{i}}_{i\leq n}$ in $\ns{A}$, where
$n\in\ns{\mathbb{N}}$, such that $x_{0}=x$, $x_{n}=y$, and $x_{i}\sim_{X}x_{i+1}$
for all $i<n$.
\item A map $f\colon X\to Y$ between standard coarse spaces is proper if
and only if $\ns{f}\left(\INF\left(X,\xi\right)\right)\subseteq\INF\left(Y,f\left(\xi\right)\right)$
for all $\xi\in X$.
\item A map $f\colon X\to Y$ between standard coarse spaces is bornologous
if and only if $x\sim_{X}y$ implies $\ns{f}\left(x\right)\sim_{Y}\ns{f}\left(y\right)$
for all $x,y\in\ns{X}$.
\end{enumerate}
Now, recall the definition of the invariant $\iota$. Let $\left(X,\xi\right)$
be a pointed coarse space in $\mathscr{C}$. Given $x,y\in\INF\left(X,\xi\right)$,
we say that $x$ and $y$ are \emph{$\iota$-equivalent} ($x\equiv_{X,\xi}^{\iota}y$)
if there exists an internal hyperfinite sequence $\set{x_{i}}_{i\leq n}$
in $\INF\left(X,\xi\right)$, called a \emph{macrochain}, such that
$x_{0}=x$, $x_{n}=y$, and $x_{i}\sim_{X}x_{i+1}$ for all $i<n$.
The quotient set
\[
\iota\left(X,\xi\right):=\INF\left(X,\xi\right)/\equiv_{X,\xi}^{\iota}
\]
can be considered as a coarsely invariant functor $\iota\colon\mathscr{C}\subseteq\pCoarse\to\Sets$.
More precisely, the morphism part of $\iota$ is given by
\[
\iota f\left[x\right]_{\equiv_{X,\xi}}:=\left[\ns{f}\left(x\right)\right]_{\equiv_{Y,\eta}}
\]
for each coarse map $f\colon\left(X,\xi\right)\to\left(Y,\eta\right)$
in $\mathscr{C}$. The well-definedness follows from the nonstandard
characterisation of coarseness mentioned above.
\begin{rem}
\citet{Gol11} introduced a similar invariant $\IPC\left(X\right)$
for a metric space $X$, where the metric function is assumed to be
finite-valued. Rougly speaking, $\IPC\left(X\right)$ is the set of
internal path components of $\INF\left(X,\xi\right)$; more precisely,
two points $x$ and $y$ of $\INF\left(X,\xi\right)$ are said to
\emph{belong to the same internal path component} if there exists
an internally continuous map $\gamma\colon\ns{\left[0,1\right]}\to\ns{X}$
such that $\gamma\left(\ns{\left[0,1\right]}\right)\subseteq\INF\left(X,\xi\right)$,
$\gamma\left(0\right)=x$ and $\gamma\left(1\right)=y$. If $X$ is
proper geodesic, then $\IPC\left(X\right)\cong\Ends\left(X\right)\cong\iota\left(X,\xi\right)$
by \citep[Lemmas 3.6 and 3.9]{Gol11}. On the other hand, if $X$
is (topologically) discrete, then $\IPC\left(X\right)\cong\INF\left(X,\xi\right)$;
hence, the coarse invariance of $\IPC$ holds only for proper geodesic
spaces.
\end{rem}

The relationship between $\sigma$ and $\iota$ is given by a natural
transformation $\omega\colon\sigma\restriction\mathscr{C}\To\iota$.
\begin{thm}[{\citep[Lemma 4.1]{Ima20a}}]
Let $\left(X,\xi\right)$ be a pointed coarse space in $\mathscr{C}$
and $s,t\in S\left(X,\xi\right)$. If $s\equiv_{X,\xi}^{\sigma}t$,
then $\ns{s}\left(i\right)\equiv_{X,\xi}^{\iota}\ns{t}\left(j\right)$
for all $i,j\in\ns{\mathbb{N}}\setminus\mathbb{N}$.
\end{thm}

\begin{cor}
Let $\left(X,\xi\right)$ be a coarse space in $\mathscr{C}$. The
map $\omega_{\left(X,\xi\right)}\colon\sigma\left(X,\xi\right)\to\iota\left(X,\xi\right)$
defined by
\[
\omega_{\left(X,\xi\right)}\left[s\right]_{\equiv_{X,\xi}^{\sigma}}:=\left[\ns{s}\left(i\right)\right]_{\equiv_{X,\xi}^{\iota}},\ i\in\ns{\mathbb{N}}\setminus\mathbb{N}
\]
is well-defined and natural in $\left(X,\xi\right)$.
\end{cor}

It has been known from \citep[Theorem 4.2]{Ima20a} that $\omega_{\left(X,\xi\right)}$
is surjective for all proper geodesic metrisable spaces $\left(X,\xi\right)$
in $\mathscr{C}$. However, no general results on the injectivity
have been given in \citep{Ima20a}. The aim of this paper is to solve
the injectivity problem for metrisable spaces. To do this, we first
observe that $\omega$ is decomposed into two natural transformations
$\varphi\restriction\mathscr{C}\colon\sigma\restriction\mathscr{C}\To\varepsilon\restriction\mathscr{C}$
and $\omega'\colon\varepsilon\restriction\mathscr{C}\To\iota$, where
$\varphi$ is the same as in \prettyref{cor:Phi} and $\omega'$ is
injective for all spaces. As a consequence, $\omega$ is injective
for all metrisable spaces and bijective for all proper geodesic metrisable
spaces. This partially answers the problems posed in \citep{Ima20a}.
Using the bijectivity result, we give a calculation of the above-mentioned
invariants of finitely branching trees in \prettyref{sec:Case-study}.
Finally, we conclude the paper with some remarks concerning non-metrisable
spaces in \prettyref{sec:Concluding-remarks}.

\section{Main Results}

We require a special case of the underspill principle.
\begin{lem}
\label{lem:weak-underspill}Let $X$ be a standard coarse space and
$\mathcal{A}$ an internal subset of $\ns{\mathcal{C}_{X}}$. If $\mathcal{A}$
contains all $E\in\ns{\mathcal{C}_{X}}$ with $\sim_{X}\subseteq E$,
then it contains $\ns{E}$ for some $E\in\mathcal{C}_{X}$.
\end{lem}

\begin{proof}
Suppose, on the contrary, that $\ns{E}\notin\mathcal{A}$ for any
$E\in\mathcal{C}_{X}$. For each $E\in\mathcal{C}_{X}$, consider
the internal set
\[
\mathcal{F}_{E}:=\set{F\in\ns{\mathcal{C}_{X}}|\ns{E}\subseteq F\notin\mathcal{A}}.
\]
Since the family $\set{\mathcal{F}_{E}|E\in\mathcal{C}_{X}}$ has
the finite intersection property, its intersection $\bigcap_{E\in\mathcal{C}_{X}}\mathcal{F}_{E}$
is non-empty by the saturation principle. Let $F$ be an element of
the intersection. Then $F\in\ns{\mathcal{C}_{X}}\setminus\mathcal{A}$
and ${\sim_{X}}\subseteq F$.
\end{proof}
The $\varepsilon$-equivalence admits the following nonstandard characterisation,
a generalisation of \citep[Lemma 3.8]{Gol11} to coarse spaces.
\begin{lem}
\label{lem:MainLem}Let $\left(X,\xi\right)$ be a pointed coarse
space in $\mathscr{C}$ and $s,t\in S\left(X,\xi\right)$. The following
are equivalent:
\begin{enumerate}
\item $s\equiv_{X,\xi}^{\varepsilon}t$;
\item $\ns{s}\left(i\right)\equiv_{X,\xi}^{\iota}\ns{t}\left(j\right)$
for all $i,j\in\ns{\mathbb{N}}\setminus\mathbb{N}$;
\item $\ns{s}\left(i\right)\equiv_{X,\xi}^{\iota}\ns{t}\left(j\right)$
for some $i,j\in\ns{\mathbb{N}}\setminus\mathbb{N}$.
\end{enumerate}
\end{lem}

\begin{proof}
$\left(1\right)\To\left(3\right)$: Let $E$ be a witness of $s\equiv_{X,\xi}^{\varepsilon}t$.
Let $B$ be an internal bounded subset of $\ns{X}$ such that $\ns{X}\setminus B\subseteq\INF\left(X,\xi\right)$
by the saturation principle. (In fact, it suffices to apply the enlargement
principle, a weaker version of saturation. See also \citep[Lemma 2.5]{Ima19}.)
By transfer, there exists an $N\in\ns{\mathbb{N}}$ such that $\set{\ns{s}\left(i\right)}_{i\geq N}$
and $\set{\ns{t}\left(i\right)}_{i\geq N}$ are contained in the same
internally $\ns{E}$-connected component of $\ns{X}\setminus B$.
Choose an $i\in\ns{\mathbb{N}}\setminus\mathbb{N}$ so that $i\geq N$.
Then there exists an internal $\ns{E}$-chain in $\ns{X}\setminus B$
that connects $\ns{s}\left(i\right)$ and $\ns{t}\left(i\right)$.
Such a chain witnesses that $\ns{s}\left(i\right)\equiv_{X,\xi}^{\iota}\ns{t}\left(i\right)$.

$\left(3\right)\To\left(2\right)$: Let $i,j\in\ns{\mathbb{N}}\setminus\mathbb{N}$
with $i\leq j$. Since $s$ is coarse, the sequence $\set{\ns{s}\left(k\right)}_{k=i}^{j}$
is a macrochain in $\INF\left(X,\xi\right)$, so $\ns{s}\left(i\right)\equiv_{X,\xi}^{\iota}\ns{s}\left(j\right)$.
Similarly, we have that $\ns{t}\left(i\right)\equiv_{X,\xi}^{\iota}\ns{t}\left(j\right)$.
The desired implication is now obvious.

$\left(2\right)\To\left(1\right)$: Suppose that $\ns{s}\left(i\right)\equiv_{X,\xi}^{\iota}\ns{t}\left(j\right)$
for some $i,j\in\ns{\mathbb{N}}\setminus\mathbb{N}$, i.e., there
exists an internal hyperfinite sequence $\set{u_{k}}_{k\leq n}$ in
$\INF\left(X,\xi\right)$ such that $u_{0}=\ns{s}\left(i\right)$,
$u_{n}=\ns{t}\left(j\right)$, and $u_{k}\sim_{X}u_{k+1}$ holds for
all $k<n$. Consider the following internal subset of $\ns{\mathcal{C}_{X}}$:
\[
\mathcal{A}:=\Set{E\in\ns{\mathcal{C}_{X}}|\begin{gathered}\left(u_{k},u_{k+1}\right)\in E\cap E^{-1}\text{ for all }k<n,\\
\left(\ns{s}\left(k\right),\ns{s}\left(k+1\right)\right)\in E\cap E^{-1}\text{ for all }k\in\ns{\mathbb{N}},\\
\left(\ns{t}\left(k\right),\ns{t}\left(k+1\right)\right)\in E\cap E^{-1}\text{ for all }k\in\ns{\mathbb{N}}
\end{gathered}
}.
\]
Since $s$ and $t$ are bornologous, $\ns{s}\left(k\right)\sim_{X}\ns{s}\left(k+1\right)$
and $\ns{t}\left(k\right)\sim_{X}\ns{t}\left(k+1\right)$ hold for
all $k\in\ns{\mathbb{N}}$. Hence $\mathcal{A}$ contains all $E\in\ns{\mathcal{C}_{X}}$
with $\sim_{X}\subseteq E$. By \prettyref{lem:weak-underspill},
there exists a standard $E\in\mathcal{C}_{X}$ such that $\ns{E}\in\mathcal{A}$.
Fix such an $E$.

Now, let $B$ be a bounded subset of $X$ containing $\xi$. Let $N=\min\set{i,j}$.
Since $s$ and $t$ are proper, $S:=\set{\ns{s}\left(k\right)}_{k\geq N}$
and $T:=\set{\ns{t}\left(k\right)}_{k\geq N}$ are contained in $\INF\left(X,\xi\right)\subseteq\ns{X}\setminus\ns{B}$.
By the choice of $E$, any two points of $S$ can be connected by
an internal $\ns{E}$-chain in $S\subseteq\ns{X}\setminus\ns{B}$;
any two points of $T$ can be connected by an internal $\ns{E}$-chain
in $T\subseteq\ns{X}\setminus\ns{B}$; and the points $\ns{s}\left(i\right)\in S$
and $\ns{t}\left(j\right)\in T$ can be connected by the internal
$\ns{E}$-chain $\set{u_{i}}_{i\leq n}$ in $\INF\left(X,\xi\right)\subseteq\ns{X}\setminus\ns{B}$.
Combining them, any two points of $S\cup T$ can be connected by an
internal $\ns{E}$-chain in $\ns{X}\setminus\ns{B}$. In other words,
there exists an $N\in\ns{\mathbb{N}}$ such that $\set{\ns{s}\left(k\right)}_{k\geq N}$
and $\set{\ns{t}\left(k\right)}_{k\geq N}$ are contained in the same
internal $\ns{E}$-connected component of $\ns{X}\setminus\ns{B}$.
By transfer, there exists an $N\in\mathbb{N}$ such that $\set{s\left(k\right)}_{k\geq N}$
and $\set{t\left(k\right)}_{k\geq N}$ are contained in the same $E$-connected
component of $X\setminus B$. Because $B$ is arbitrary and $E$ does
not depend on $B$, we have that $s\equiv_{X,\xi}^{\varepsilon}t$.
\end{proof}
\begin{rem}
The use of the saturation principle (\prettyref{lem:weak-underspill})
in the implication $\left(2\right)\To\left(1\right)$ is avoidable
if $\left(X,\xi\right)$ is metrisable. Suppose that $\ns{s}\left(i\right)\equiv_{X,\xi}^{\iota}\ns{t}\left(j\right)$
for some $i,j\in\ns{\mathbb{N}}\setminus\mathbb{N}$, i.e., there
exists an internal hyperfinite sequence $\set{u_{k}}_{i\leq n}$ in
$\INF\left(X,\xi\right)$ such that $u_{0}=\ns{s}\left(i\right)$,
$u_{k}=\ns{t}\left(j\right)$, and $\ns{d_{X}}\left(u_{k},u_{k+1}\right)$
is finite for all $k<n$. The maximum $L:=\max_{k<n}\ns{d_{X}}\left(u_{k},u_{k+1}\right)$
exists by transfer, and is finite. Since $s$ and $t$ are bornologous,
the suprema $M:=\sup_{k\in\mathbb{N}}d_{X}\left(s\left(k\right),s\left(k+1\right)\right)$
and $N:=\sup_{k\in\mathbb{N}}d_{X}\left(k\left(i\right),k\left(i+1\right)\right)$
exist. Fix a standard $K>0$ such that $L,M,N\leq K$ and let $E:=\set{\left(x,y\right)\in X\times X|d_{X}\left(x,y\right)\leq K}\in\mathcal{C}_{X}$.
The remaining proof is the same as above.
\end{rem}

\begin{thm}
\label{thm:omega-prime}Let $\left(X,\xi\right)$ be a pointed coarse
space in $\mathscr{C}$. The map $\omega_{\left(X,\xi\right)}'\colon\varepsilon\left(X,\xi\right)\to\iota\left(X,\xi\right)$
defined by
\[
\omega_{\left(X,\xi\right)}'\left[s\right]_{\equiv_{X,\xi}^{\varepsilon}}:=\left[\ns{s}\left(i\right)\right]_{\equiv_{X,\xi}^{\iota}},\ i\in\ns{\mathbb{N}}\setminus\mathbb{N}
\]
is well-defined, injective and natural in $\left(X,\xi\right)$.
\end{thm}

\begin{proof}
The well-definedness follows from the implication $\left(1\right)\To\left(2\right)$
and the injectivity does from the implication $\left(3\right)\To\left(1\right)$
of \prettyref{lem:MainLem}. The naturality is trivial by definition.
\end{proof}
\begin{thm}
\label{thm:omega-split}$\omega=\omega'\circ\left(\varphi\restriction\mathscr{C}\right)$.
\end{thm}

\begin{proof}
Trivial.
\end{proof}
\begin{thm}
\label{thm:MainThm}The map $\omega_{\left(X,\xi\right)}\colon\sigma\left(X,\xi\right)\to\iota\left(X,\xi\right)$
is injective for all metrisable spaces $\left(X,\xi\right)$ in $\mathscr{C}$.
\end{thm}

\begin{proof}
Combine \prettyref{cor:phi-bijective}, \prettyref{thm:omega-prime}
and \prettyref{thm:omega-split}.
\end{proof}
This answers \citep[Problems 5.3 and 5.5]{Ima20a} for metrisable
spaces, and gives a partial answer to \citep[Problem 5.4]{Ima20a}.
\begin{cor}
The map $\omega_{\left(X,\xi\right)}\colon\sigma\left(X,\xi\right)\to\iota\left(X,\xi\right)$
is bijective for all proper geodesic metrisable spaces $\left(X,\xi\right)$
in $\mathscr{C}$.
\end{cor}

\begin{cor}
Let $\left(X,\xi\right)$ be a proper geodesic metrisable space in
$\mathscr{C}$. Then $\sigma\left(X,\xi\right)\cong\varepsilon\left(X,\xi\right)\cong\iota\left(X,\xi\right)\cong\Ends\left(X\right)$.
\end{cor}

\section{\label{sec:Case-study}Case study}

Throughout this section, we identify a graph with its geometric realisation.
We also assume that $\mathscr{C}$ contains (asymorphic copies of)
all pointed coarse spaces we consider. Let $G:=\left(V,E\right)$
be a locally finite connected graph. The metric $d_{G}\colon V\times V\to\mathbb{R}_{\geq0}$
is defined as usual:
\[
d_{G}\left(v,w\right):=\text{the length of the shortest path between }v\text{ and }w.
\]
This makes the graph $G$ a proper geodesic metric space where each
edge is isometric to the unit interval $\left[0,1\right]$. The vertex
set $V$ is coarsely equivalent to the whole graph $G$, so we obtain
the isomorphisms:
\[
\xymatrix{ & \sigma\left(V,v\right)\ar[d]_{\cong}\ar[r]_{\cong}^{\varphi} & \varepsilon\left(V,v\right)\ar[d]_{\cong}\ar[r]_{\cong}^{\omega'} & \iota\left(V,v\right)\ar[d]_{\cong}\\
\Ends\left(G\right)\ar[r]_{\cong} & \sigma\left(G,v\right)\ar[r]_{\cong}^{\varphi} & \varepsilon\left(G,v\right)\ar[r]_{\cong}^{\omega'} & \iota\left(G,v\right)
}
\]
where the vertical maps are induced by the inclusion $V\hookrightarrow G$.
\begin{thm}
Let $T$ be a finitely branching tree with root $r$. Then $\iota\left(T,r\right)$
is equipotent to the set of all infinite branches of $T$, and so
are $\sigma\left(T,r\right)$, $\varepsilon\left(T,r\right)$ and
$\Ends\left(T\right)$.
\end{thm}

\begin{proof}
We denote the set of all infinite branches by $\left[T\right]$. It
suffices to prove that $\left[T\right]\cong\iota\left(T,r\right)$.
Define a map $\psi\colon\left[T\right]\to\iota\left(T,r\right)$ as
follows. Let $f\in\left[T\right]$. Choose an arbitrary $i\in\ns{\mathbb{N}}\setminus\mathbb{N}$.
By transfer, $\ns{f}\left(i\right)\in\ns{T}$ and $d_{T}\left(r,\ns{f}\left(i\right)\right)=i$,
so $\ns{f}\left(i\right)\in\INF\left(T,r\right)$. Now, define $\psi\left(f\right):=\left[\ns{f}\left(i\right)\right]_{\equiv_{T,r}^{\iota}}$,
where the right hand side is independent of the choice of $i$ by
the following fact.
\begin{claim*}
$\ns{f}\left(i\right)\equiv_{T,r}^{\iota}\ns{f}\left(j\right)$ for
all $i,j\in\ns{\mathbb{N}}\setminus\mathbb{N}$.
\end{claim*}

\subsubsection*{Proof}

We may assume without loss of generality that $i\leq j$. The sequence
$\set{\ns{f}\left(k\right)}_{k=i}^{j}$ witnesses that $\ns{f}\left(i\right)\equiv_{T,r}^{\iota}\ns{f}\left(j\right)$.
\begin{claim*}
$\psi$ is surjective.
\end{claim*}

\subsubsection*{Proof}

Let $x\in\INF\left(T,r\right)$. By transfer, there exists a (unique)
hyperfinite sequence $\set{t_{i}}_{i\leq n}$ in $\ns{T}$ such that
$t_{0}=r$, $t_{n}=x$ and $t_{i+1}$ is a child of $t_{i}$ for each
$i<n$. Notice that $\ns{d_{T}}\left(r,t_{i}\right)=i$ and $\ns{d_{T}}\left(r,x\right)=n\in\ns{\mathbb{N}}\setminus\mathbb{N}$.
For each (standard) $i\in\mathbb{N}$, since $T$ is finitely branching,
the set $T_{i}:=\set{t\in T|d_{T}\left(r,t\right)=i}$ is finite,
so $t_{i}\in\ns{T_{i}}=T_{i}$. Hence the map $f\colon\mathbb{N}\to T$
defined by $f\left(i\right):=t_{i}$ is an infinite branch of $T$.

Let us verify that $\psi\left(f\right)=\left[x\right]_{\equiv_{X,\xi}^{\iota}}$.
First note that $\ns{f}\left(i\right)=t_{i}$ holds for all $i\in\mathbb{N}$
by definition. By the countable saturation, $\ns{f}\left(i\right)=t_{i}$
holds for some (infinite) $i\in\ns{\mathbb{N}}\setminus\mathbb{N}$
with $i\leq n$. The sequence $\set{t_{k}}_{k=i}^{n}$ witnesses that
$\ns{f}\left(i\right)\equiv_{T,r}^{\iota}x$. Therefore $\psi\left(f\right)=\left[x\right]_{\equiv_{T,r}^{\iota}}$.
See also the proof of \citep[Theorem 4.2]{Ima20a}.
\begin{claim*}
$\psi$ is injective.
\end{claim*}

\subsubsection*{Proof}

Let $\set{t_{i}}_{i\leq n}$ be a macrochain in $\INF\left(T,r\right)$
and $n\in\mathbb{N}$. Any two adjacent nodes of $\set{t_{i}}_{i\leq n}$
have the same ancestor of level $n$. (Otherwise, the distance $\ns{d_{T}}\left(t_{i},t_{i+1}\right)$
would be infinite, which is a contradiction.) By the transferred induction
principle, all nodes of $\set{t_{i}}_{i\leq n}$ have the same ancestor
of level $n$.

Now, let $f,g\in\left[T\right]$ and suppose $\psi\left(f\right)=\psi\left(g\right)$,
i.e., there exists a macrochain in $\INF\left(T,r\right)$ connecting
$\ns{f}\left(i\right)$ and $\ns{g}\left(i\right)$ for $i\in\ns{\mathbb{N}}\setminus\mathbb{N}$.
All nodes of the macrochain have the same ancestors of standard level.
In particular, $\ns{f}\left(i\right)$ and $\ns{g}\left(i\right)$
have the same ancestors of standard level. In other words, $f\left(j\right)=g\left(j\right)$
for all standard $j\in\mathbb{N}$. Hence $f=g$.\qedhere

\end{proof}
\begin{example}
Let $G$ be a finitely generated group endowed with a finite generating
set $S$. The Cayley graph $\Cay_{S}\left(G\right)$ is locally finite
and connected; hence $\varepsilon\left(G,e_{G}\right)\cong\varepsilon\left(G,e_{G}\right)\cong\iota\left(G,e_{G}\right)\cong\Ends\left(\Cay_{S}\left(G\right)\right)=:\Ends\left(G\right)$.
\begin{enumerate}
\item Consider the Abelian group $\mathbb{Z}$ with the generating set $\set{\pm1}$.
The infinite part $\INF\left(\mathbb{Z},0\right)$ consists of positive
and negative infinite hypernatural numbers. Two infinite points are
$\iota$-equivalent if and only if they have the same sign. To see
the ``only if'' part, let $\set{x_{i}}_{i\leq n}$ be a macrochain
in $\INF\left(\mathbb{Z},0\right)$. The internal set $A:=\set{i\in\ns{\mathbb{N}}|\sgn\left(x_{i}\right)=\sgn\left(x_{0}\right)\text{ or }i>n}$
contains $0$, and is closed under successor. By the transferred induction
principle, $A=\ns{\mathbb{N}}$, i.e., the components of $\set{x_{i}}_{i\leq n}$
have the same sign. On the other hand, the ``if'' part is evident.
Thus we have that $\left|\sigma\left(\mathbb{Z},0\right)\right|=\left|\varepsilon\left(\mathbb{Z},0\right)\right|=\left|\iota\left(\mathbb{Z},0\right)\right|=\left|\Ends\left(\mathbb{Z}\right)\right|=2$.
In fact, the Cayley graph of $\mathbb{Z}$ with respect to $\set{\pm1}$
can be considered as a tree, where the root has exactly two children,
and each node other than the root has exactly one child. Such a tree
has exactly two infinite branches.
\item Let $S$ be a finite set of cardinality $\geq2$ and $F_{S}$ the
free group generated by $S$. The Cayley graph $\Cay_{S\cup S^{-1}}\left(F_{S}\right)$
is a finitely branching tree with the root $e_{F_{S}}$. The root
has $2\left|S\right|$ children and each node other than the root
has $2\left|S\right|-1$ children. So $\Cay_{S\cup S^{-1}}\left(F_{S}\right)$
has exactly $2^{\aleph_{0}}$ infinite branches. Hence $\left|\sigma\left(F_{S},e_{F_{S}}\right)\right|=\left|\varepsilon\left(F_{S},e_{F_{S}}\right)\right|=\left|\iota\left(F_{S},e_{F_{S}}\right)\right|=\left|\Ends\left(F_{S}\right)\right|=2^{\aleph_{0}}$.
\end{enumerate}
\end{example}

The following examples cannot be obtained as a Cayley graph of a finitely
generated group.
\begin{example}
Let $\Sigma:=\set{a,b}$ be an alphabet. The set $\Sigma^{<\omega}$
of all finite words over $\Sigma$ can be considered as a binary tree
with the root $\emptyset$. Similarly to the above example $F_{S}$,
we have that $\left|\sigma\left(\Sigma^{<\omega},\emptyset\right)\right|=\left|\varepsilon\left(\Sigma^{<\omega},\emptyset\right)\right|=\left|\iota\left(\Sigma^{<\omega},\emptyset\right)\right|=\left|\Ends\left(\Sigma^{<\omega}\right)\right|=2^{\aleph_{0}}$.
On the other hand, the subtree
\[
T:=\set{a^{n}b^{m}\in\Sigma^{<\omega}|n,m\in\mathbb{N}}
\]
has exactly $\aleph_{0}$ infinite branches. Hence $\left|\sigma\left(T,\emptyset\right)\right|=\left|\varepsilon\left(T,\emptyset\right)\right|=\left|\iota\left(T,\emptyset\right)\right|=\left|\Ends\left(T\right)\right|=\aleph_{0}$.
This affirmatively answers \citep[Problem 5.2]{Ima20a}, which asks
if $\iota\left(X,\xi\right)$ is countably infinite for some pointed
coarse space $\left(X,\xi\right)$ in $\mathscr{C}$ (provided that
$\mathscr{C}$ is sufficiently rich).
\end{example}

\section{\label{sec:Concluding-remarks}Concluding remarks}

In summary, the coarsely invariant functors $\sigma\restriction\mathscr{C},\varepsilon\restriction\mathscr{C},\iota\colon\mathscr{C}\subseteq\pCoarse\to\Sets$
are related with the following commutative diagram of natural transformations:

\[
\xymatrix{\sigma\restriction\mathscr{C}\ar@{=>}[dr]_{\varphi\restriction\mathscr{C}}\ar@{=>}[rr]^{\omega} &  & \iota\\
 & \varepsilon\restriction\mathscr{C}\ar@{=>}[ur]_{\omega'}
}
\]
where $\varphi\colon\sigma\To\varepsilon$ is bijective on $\pMetr$,
$\omega'$ injective on $\mathscr{C}$, and $\omega$ injective on
$\pMetr\cap\mathscr{C}$. Moreover, $\omega'$ and $\omega$ are bijective
for proper geodesic metrisable spaces in $\mathscr{C}$. This enables
us to calculate, for a large class of spaces, the invariants $\sigma$,
$\varepsilon$ and $\Ends$ in an intuitive way, as we demonstrated
in \prettyref{sec:Case-study}.

Our proof of the injectivity of $\omega$ does not apply to general
pointed coarse spaces, because the proof of the implication from $\equiv_{X,\xi}^{\varepsilon}$
to $\equiv_{X,\xi}^{\sigma}$ depends on the metrisability as we pointed
out in \prettyref{rem:remark-on-phi}. Hence, the following problem
remains open.
\begin{problem}
Is the map $\omega_{\left(X,\xi\right)}\colon\sigma\left(X,\xi\right)\to\iota\left(X,\xi\right)$
injective for all non-metrisable spaces $\left(X,\xi\right)$ in $\mathscr{C}$?
\end{problem}

This is equivalent to the following purely standard problem.
\begin{problem}
Is the map $\varphi_{\left(X,\xi\right)}\colon\sigma\left(X,\xi\right)\to\varepsilon\left(X,\xi\right)$
injective for all non-metrisable spaces $\left(X,\xi\right)$ in $\mathscr{C}$?
\end{problem}

\bibliographystyle{IEEEtranSN}
\bibliography{Injectivity}

\end{document}